\documentclass[11pt,a4paper,draft,twoside,reqno]{amsart}
\usepackage{amsfonts,amssymb,euscript,amscd}
\DeclareFontFamily{OML}{cyr}{} \DeclareFontShape{OML}{cyr}{m}{n}{
  <5> <6> <7> <8> <9> gen * wncyr
  <10> <10.95> <12> <14.4> <17.28> <20.74> <24.88> wncyr10
  }{}
\DeclareSymbolFont{rusletters}{OML}{cyr}{m}{n}
\DeclareSymbolFontAlphabet{\rusmath}{rusletters}
\DeclareMathSymbol\re{\rusmath}{rusletters}{"03}
\newtheorem{theorem}{Theorem}[section]
\newtheorem{proposition}[theorem]{Proposition}
\newtheorem{lemma}[theorem]{Lemma}
\newtheorem{corollary}[theorem]{Corollary}

\newcommand*{\im}{\mathop{\rm Im}\nolimits}
\newcommand*{\id}{\mathop{\rm id}\nolimits} % identical map
\newcommand*{\Orb}{\mathop{\rm Orb}\nolimits}
\newcommand*{\Vect}{\mathop{\rm Vect}\nolimits}
\newcommand{\0}{\mbox{{\bf 0}}}
\newcommand*{\E}{{\EuScript E}} % Equation
\newcommand*{\A}{{\EuScript A}} % Isotropy space
\newcommand*{\Hr}{{\EuScript H}} % Horisontal space
\newcommand*{\R}{\mathop{\mathbb R}\nolimits} % Real numbers
\newcommand*{\g}{\mathop{\mathfrak g}\nolimits} % Lie algbr.
\newcommand*{\eq}{y''=u^0(x,y)+u^1(x,y)y'+u^2(x,y)(y')^2
  +u^3(x,y)(y')^3}
\title[On the obstruction to linearizability]{On the obstruction to
  linearizability of 2-order ordinary differential equations}
\author[V.A.Yumaguzhin]{Valeriy A. Yumaguzhin}
\date{29 November 2002}
\address{Program Systems Institute, m. Botik, Pereslavl'-Zalesskiy,
  152020, Russia}
\email{yuma@diffiety.botik.ru}
\begin{document}
\begin{abstract} In this paper, we investigate the action of
  pseudogroup of all point transformations on the natural bundle of
  equations
  $$
    \eq\,.
  $$
  We calculate the 1-st nontrivial differential invariant of this
  action. It is a horizontal differential 2-form with values in
  some algebra, it is defined on the bundle of 2--jets of sections
  of the considered bundle. We prove that this form is a unique
  obstruction to linearizability of these equations by point
  transformations.
\end{abstract}
\maketitle
%
%%%%%%%%%%%%%%%%%%%%%%%%%%%%%%%%%%%%%%%%%%%%%%%%%%%%%%%%%%%%%%%%%%
\section{Introduction}
It is well known that any point transformation takes a 2-order
linear ordinary differential equation to an equation of the form
\begin{equation}\label{eq}
  \eq\,.
\end{equation}
By $\pi$ we denote the natural bundle of equations \eqref{eq}.

It is well known that any point transformation takes an arbitrary
equation \eqref{eq} to the equation of the same form. This means
that the pseudogroup $\Gamma$ of all point transformations acts on
$\pi$. This action can be lifted in the natural way to the action
on the bundle $J^k\pi$ of $k$--jets of sections of
$\pi\,,\;k=1,2\ldots$

In this paper, we investigate these actions. Earlier in
\cite{GYum}, we obtained the following:

1. $J^k\pi$ is an orbit of the action of $\Gamma$ iff $k=0,1$,

2. $J^2\pi$ is divided into two orbits of the action
$J^2\pi=\Orb_1\cup\Orb_2$ with $\dim\Orb_1=\dim J^k\pi$ and
$\dim\Orb_2=\dim J^k\pi - 2$,

3. Equation \eqref{eq} can be reduced to the linear form by a
point transformation iff the collection of its coefficients is a
solution of the equations defining the submanifold $\Orb_2$ in
$J^2\pi$.

\noindent This means that the first nontrivial differential
invariant of the actions of $\Gamma$ "lives" on $J^2\pi$ and it is
a unique obstruction to the linearizability of equations
\eqref{eq} by point transformations.

The aim of this paper is to construct this obstruction. We
constructed it in subsection \ref{SubSctObstrFrm} of this paper.
It is a horizontal differential 2-form on $J^2\pi$ with values in
some algebra. This form is nontrivial at any point of $\Orb_1$ and
it is zero at any point of $\Orb_2$.

Recall that in \cite{Crtn}, Cartan proved that equation \eqref{eq}
is equivalent to some projective connection and the equation can
be reduced to the linear form by a point transformation iff the
curvature form of this connection is equal to zero. We do not use
projective connections to construct the obstruction form. Our
construction recall the well known construction of structure
functions of prolongations of $G$--structures (see
\cite{Strnbrg}).\\

Below, all manifolds and maps are supposed to be smooth. By
$[f]_p^k$ denote the $k$--jet of the map $f$ at the point $p$, by
$\R$ denote the field of real numbers, and by $\R^n$ denote the
$n$--dimensional arithmetic space.

%%%%%%%%%%%%%%%%%%%%%%%%%%%%%%%%%%%%%%%%%%%%%%%%%%%%%%%%%%%%%%%%%%
\section{The natural bundle of equations}
%
%=================================================================
\subsection{Liftings of point transformations}
%.................................................................
\subsubsection{The lifting to the bundle of equations}
Let $$
  \pi :E=\R^2\times\R^4\to\R^2
$$ be a product bundle. By $x^1,x^2$ denote the standard
coordinate on the base of $\pi$, by $u^0, u^1, u^2, u^3$ denote
the standard coordinates on the fiber of $\pi$.

Let $\E$ be an arbitrary equation \eqref{eq}. We identify $\E$
with the section $S_{\E}$ of $\pi$ defined by the formula $$
  S_{\E}:(x^1,\,x^2)\mapsto
  \bigl(\,x^1,\,x^2,\,u^0(x^1,x^2),\,u^1(x^1,x^2),\,u^2(x^1,x^2),\,
  u^3(x^1,x^2)\,\bigr).
$$ Clearly, this identification is a bijection between the set of
all equations \eqref{eq} and the set of all sections of $\pi$.

It is well known (see \cite{Arnd}) that an arbitrary point
transformation
\begin{equation}\label{PntTr}
  f:(x^1,x^2)\mapsto\bigl(\,\tilde x^1=f^1(x^1,x^2),\;
  \tilde x^2=f^2(x^1,x^2)\,\bigr)\,.
\end{equation}
transforms an equation of form \eqref{eq} to the equation of the
same form. The coefficients of the obtained equation are expressed
in terms of the coefficients of the initial one and the
derivatives of order $\leq 2$ of the inverse transformation to
$f$:
\begin{gather}
  \tilde u^{\alpha}=\Phi^{\alpha}\Bigl(\,u^{\beta},\,
  \frac{\partial g^i}{\partial\tilde x^j},\,
  \frac{\partial^{\,2}g^i}{\partial\tilde x^{j_1}\partial\tilde x^{j_2}}\,\Bigr)
  \,,\label{CffTr}\\
  \alpha,\beta=0,1,2,3\,,\;g=\bigl(g^1,\,g^2\bigr)
  =f^{-1}\,,\;i,j,j_1,j_2=1,2\,.\notag
\end{gather}

Equations \eqref{PntTr} and \eqref{CffTr} defines the
diffeomorphism $f^{(0)}$ of the bundle $\pi$ which is called {\it
the lifting of $f$ to the bundle $\pi$}.

Obviously, the following diagram $$
  \begin{CD}
    E        @>f^{(0)}>> E\\
    @V\pi VV             @VV\pi V\\
    \R       @>>f>       \R
  \end{CD}
$$ is commutative (in the domain of $f^{(0)}$).

For any point transformation $f$, we define the transformation of
sections of $\pi$ by the formula
\begin{equation}\label{TrnsftnSctn}
  S\mapsto f(\,S\,)=f^{(0)}\circ S\circ f^{-1}\,.
\end{equation}

Equations \eqref{CffTr} can be represented now as $$
    S_{\tilde\E}=f(\,S_{\E}\,)\,.
$$ Now the following statement is obvious.
\begin{proposition}\label{RdcJrms}
  Let $\E\,,\;\tilde\E$ be equations of form \eqref{eq}. Then a
  point transformation $f$ takes $\E$ to $\tilde\E$ iff
  $S_{\tilde\E}=f(\,S_{\E}\,)$.
\end{proposition}

%.................................................................
\subsubsection{The lifting to jet bundles}
By $[S]_p^k$ denote the $k$--jet of a section $S$ of $\pi$ at the
point $p\,,\;k=0,1,2,\ldots,\infty$. By $$
  \pi_k:J^k\pi\to\R^2\,,\;\pi_k:[S]_p^k\mapsto p\,,
$$ denote the bundle of all $k$--jets of sections of $\pi$. The
projection\linebreak $\pi_{k,r}:J^k\pi\to J^r\pi\,,\;k>r$, is
defined by $\pi_{k,r}(\,[S]_p^k\,)=[S]_p^r$. By  definition, put
$J_p^k\pi=\pi_k^{-1}(p)$.

Every section $S$ of $\pi$ generates the section $j_kS$ of the
bundle $\pi_k$ by the formula $j_kS: p\;\mapsto\;[S]_p^k$.

By $x^1,x^2,u^i_{\sigma}\,,\;i=0,\ldots,3\,,\;0\leq|\sigma|\leq
k$, denote the standard coordinates in $J^k\pi$, here $\sigma$ is
the multi-index $\{j_1\ldots j_r\}\,,\;|\sigma|=r\,,\;j_1,
\ldots,j_r=1,2$. By definition, put $\sigma j=\{j_1\ldots j_rj\}$

Any point transformation $f$ can be lifted to the diffeomorphism
$f^{(k)}$ of $J^k\pi$ by the formula
\begin{equation}\label{LftTr}
  f^{(k)}(\,[S]^k_p\,)=\bigl[\,f^{(0)}\circ S\circ
  f^{-1}\,\bigr]^k_{f(p)}\,.
\end{equation}
The diffeomorphism $f^{(k)}$ is called {\it the lifting of $f$ to
the jet bundle $J^k\pi$}.

Obviously, for any $l>m$, the diagram $$
  \begin{CD}
    J^l\pi         @>f^{(l)}>> J^l\pi\\
    @V\pi_{l,m}VV             @VV\pi_{l,m} V\\
    J^m\pi         @>>f^{(m)}> J^m\pi
  \end{CD}
$$ is commutative (in the domains of $f^{(l)}$).

By $\Gamma$ we denote the pseudogroup of all point transformation
of the base of $\pi$, by $\Gamma^{(k)}$ we denote the
transformation pseudogroup in $J^k\pi$ generated by all
diffeomorphisms $f^{(k)}\,,\;f\in\Gamma$.

%=================================================================
\subsection{Liftings of vector fields}
Let $X$ be a vector field in the base of $\pi$ and let $f_t$ be
its flow. Then the flow $f_t^{(k)}$ in $J^k\pi$ defines the vector
field $X^{(k)}$ in $J^k\pi$ which is called {\it the lifting of
$X$ to $J^k\pi$}. Obviously
\begin{equation}\label{kPrlng}
  (\,\pi_{l,m}\,)_*\bigl(\,X^{(l)}\,\bigr)=X^{(m)}\,,\;\;
  \infty\geq l>m\geq -1\,,
\end{equation}
where $X^{(-1)}=X$.

Let $$
  X=X^1(x^1,x^2)\frac{\partial}{\partial x^1}
  +X^2(x^1,x^2)\frac{\partial}{\partial x^2}\,,
$$ then we have the following formula (see \cite{KV})
\begin{equation}\label{InftPrlng}
  X^{(\infty)}=X^1D_1+X^2D_2+\re_{\psi(X)}\,,
\end{equation}
where $$
  D_j=\frac{\partial}{\partial x^j}
      +\sum_{|\sigma|\geq 0}\sum_{i=0}^{3}u^i_{\sigma j}
      \frac{\partial}{\partial u^i_{\sigma}}\,,
$$ is the operator of total derivation w.r.t. $x^j$,
\begin{equation}\label{Evol}
 \re_{\psi(X)}=\sum_{|\sigma|\geq 0}\sum_{i=0}^{3}
 D_{\sigma}\bigl(\,\psi^i(X)\,\bigr)
 \frac{\partial}{\partial u^i_{\sigma}}
\end{equation}
is the operator of evolution differentiation corresponding to the
generating function $\psi(X)=(\psi^0(X),\ldots,\psi^3(X))^t$,
$\sigma =\{j_1\ldots j_r\}\,,\;D_{\sigma} =D_{j_1}\circ\ldots\circ
D_{j_r}$. The function $\psi(X)$ is defined in the following way.
Let $S$ be a section of $\pi$ defined in the domain of $X$, let
$\theta_1=[S]^1_p$, and let $p=\pi_1(\theta_1)$; then
\begin{equation}\label{DfrmtnVlst0}
  \psi(X)(\theta_1)
  =\begin{pmatrix}
     \psi^0(X)(\theta_1)\\
     \cdots\\
     \psi^3(X)(\theta_1)\,.
  \end{pmatrix}
  =\frac{d}{dt}(\,f_t^{(0)}\circ S\circ
  f_t^{-1}\,)\Bigr|_{t=0}(p)
\end{equation}
Obviously, $\psi(X)(\theta_1)$ is the deformation velocity of the
section $S$ at the point $p$ under the action of the flow $f_t$.

Let $\theta_1=(\,x^1, x^2, u^i, u^i_j\,),\,i=0,1,2,3,\,j=1,2$;
then it can be calculated that
\begin{equation}\label{DfrmtnVlst1}
  \psi(X)(\theta_1)=
  \begin{pmatrix}
    -u^0_1X^1-u^0_2X^2\\
    - 2u^0X^1_1 + u^0X^2_2 - u^1X^2_1
    + X^2_{11}\vspace{.1in}\\
    -u^1_1X^1-u^1_2X^2\\
    - 3u^0X^1_2 - u^1X^1_1 - 2u^2X^2_1
    - X^1_{11} + 2X^2_{12}\vspace{.1in}\\
    -u^2_1X^1-u^2_2X^2\\
    - 2u^1X^1_2 - u^2X^2_2 - 3u^3X^2_1
    - 2X^1_{12} + X^2_{22}\vspace{.1in}\\
    -u^3_1X^1-u^3_2X^2\\
    -  u^2X^1_2 + u^3X^1_1 - 2u^3X^2_2
    - X^1_{22}
  \end{pmatrix}\,,
\end{equation}
where $X^i_j=\displaystyle\frac{\partial X^i}{\partial x^j}(p)$
and $X^i_{j_1j_2}=\displaystyle\frac{\partial^2X^i}{\partial
x^{j_1}\partial x^{j_2}}(p)$\,.
\medskip

Let $\Vect\R^2$ and $\Vect J^k\pi$ be the Lie algebras of all
vector fields in $\R^2$ and $J^k\pi$ respectively.
\begin{proposition} The map
  $$
    \Vect\R^2\to\Vect J^k\pi\,,\quad X\mapsto X^{(k)}\,,
  $$
  is a Lie algebra homomorphism.
\end{proposition}
\begin{proof}
  The map $\Gamma\to\Gamma^{(k)}\,,\;f\mapsto f^{(k)}$, is a
  homomorphism of Lie pseu\-do\-gro\-ups. It has as a consequence the
  statement of the proposition. Indeed, let $X\,,Y$ be vector
  fields on $\R^2$ and let $f_t\,,g_s$ be their flows
  respectively. Then
  \begin{multline*}
    [\,X^{(k)},\,Y^{(k)}\,]=\lim_{t\to 0}\frac{1}{t}
    \Bigl(\,Y^{(k)}-(f^{(k)}_t)_*(Y^{(k)}\circ
    f^{(k)}_{-t})\,\Bigr)\\
    =\lim_{t\to 0}\frac{1}{t}
    \Bigl(\,\frac{d}{ds}\Bigl|_{s=0}g^{(k)}_s
    -(f^{(k)}_t)_*\bigl(\frac{d}{ds}\Bigl|_{s=0}g^{(k)}_s
    \circ f^{(k)}_{-t}\bigr)\,\Bigr)
    =\lim_{t\to 0}\frac{1}{t}
    \Bigl(\,\frac{d}{ds}\Bigl|_{s=0}g^{(k)}_s\\
    -\frac{d}{ds}\Bigl|_{s=0}f^{(k)}_t\circ\,g^{(k)}_s
    \circ f^{(k)}_{-t}\,\Bigr)
    =\lim_{t\to 0}\frac{1}{t}\frac{d}{ds}\Bigl|_{s=0}
    \Bigl(\,g^{(k)}_s\circ f^{(k)}_t\circ\,g^{(k)}_s
    \circ f^{(k)}_{-t}\,\Bigr)\\
    =\lim_{t\to 0}\frac{1}{t}\frac{d}{ds}\Bigl|_{s=0}
    \Bigl(\,g_s\circ f_t\circ\,g_s
    \circ f_{-t}\,\Bigr)^{(k)}
    =\lim_{t\to 0}\frac{1}{t}
    \Bigl(\,\frac{d}{ds}\Bigl|_{s=0}g_s\\
    -\frac{d}{ds}\Bigl|_{s=0}f_t\circ\,g_s
    \circ f_{-t}\,\Bigr)^{(k)}
    =\lim_{t\to 0}\frac{1}{t}
    \bigl(\,Y-(f_t)_*(Y\circ f_{-t})\,\bigr)^{(k)}
    =[\,X,\,Y\,]^{(k)}\,.
  \end{multline*}
  The $\R$ -- linearity of the map $ X\mapsto X^{(k)}$ is obvious.
\end{proof}

%%%%%%%%%%%%%%%%%%%%%%%%%%%%%%%%%%%%%%%%%%%%%%%%%%%%%%%%%%%%%%%%%%
\section{Isotropy algebras and spaces}
%=================================================================
\subsection{Preliminaries}
In this subsection, we recall some necessary notions concerning
formal vector fields, prolongations of subspaces, Spenser
cohomologies and the decomposition of tangent spaces to $J^k\pi$
(see \cite{BrnshtnRznfld}, \cite{GllmnStrnbrg}, and \cite{KV}).

%.................................................................
\subsubsection{Formal vector fields}
By $W_p$ we denote the Lie algebra of $\infty$--jets at $p\in\R^2$
of all vector fields defined in a neighborhoods of $p$. Recall
that the structure of Lie algebra on $W_p$ is defined by the
operations
\begin{gather*}
  \lambda [X]^{\infty}_p\stackrel{df}{=}[\lambda X]^{\infty}_p\,,\quad
  [X]^{\infty}_p+[Y]^{\infty}_p\stackrel{df}{=}[X+Y]^{\infty}_p\,,\\
  \bigl[\,[X]^{\infty}_p,[Y]^{\infty}_p\,\bigr]\stackrel{df}{=}
  \bigl[\,[X,Y]\,\bigr]^{\infty}_p\\
  \forall\;\lambda\in\R\,,\quad
  \forall\;\;[X]^{\infty}_p\,,[Y]^{\infty}_p\in W_p\,.
\end{gather*}

By $L_p^k\,,\;k=-1,0,1,2,\ldots$, we denote the subalgebra in
$W_p$ defined by $$
  L_p^k=\bigl\{\,[X]^{\infty}_p\in W_n\,\bigl|
  \,[X]^k_p=0\,\bigr\}\,,\;k\geq 0\,,
  \quad L_p^{-1}=W_p\,.
$$ By definition, put $$
  V_p=W_p/L_p^0\,.
$$ Obviously, $V_p\cong T_p\R^2$. We have the filtration $$
  W_p=L_p^{-1}\supset L_p^0\supset L_p^1\supset\ldots\supset
  L_p^k\supset L_p^{k+1}\supset\ldots\,.
$$

For any $i>j\geq 0$, we denote by $\rho_{i,j}$ the natural
projection $$
  \rho_{i,j}:W_p/L_p^i\to W_p/L_p^j\,,\quad
  \rho_{i,j}:[X]_p^i\mapsto [X]_p^j
$$ and by definition, put $$
  \rho_i = \rho_{i,0}\,.
$$

Taking into account that $$
  [\,L_p^i\,,\;L_p^j\,]\
  ,=\,L_p^{i+j}\,,\quad
  i,j=-1,0,1,2,\ldots\,,
$$ we see that the bracket operation $[\,\cdot\,,\cdot\,]$ on
$W_p$ generates the following maps
\begin{align}
  [\,\cdot\,,\,\cdot\,]:\,&W_p/L_p^k\times W_p/L_p^k\to W_p/L_p^{k-1}\,,\label{brkt1}\\
  [\,\cdot\,,\,\cdot\,]:\,&V_p\times L_p^k/L_p^{k+1}\to L_p^{k-1}/L_p^k\,.\label{brkt2}
\end{align}
The last map generates the isomorphism $$
  L_p^k/L_p^{k+1}\cong V_p\otimes S^k(V_p^*)\,.
$$

Let $g_k$ be a subspace of $L_p^{k-1}/L_p^k$. The subspace
$g_k^{(1)}\subset L_p^k/L_p^{k+1}$ defined by $$
  g_k^{(1)}=\bigl\{\,X\in L_p^k/L_p^{k+1}\,\bigl|
  \,[\,v\,,\,X\,]\in g_k\;\;\forall\;v\in V_p\,\bigr\}
$$ is called {\it the 1-st prolongation of $g_k$}.

Suppose the sequence of subspaces $$
  g_1\,,\;g_2\,,\;\ldots\,,\;g_i\,,\;\ldots
$$ satisfies to the property $$
  [\,V\,,\;g_{i+1}\,]\subset g_i\,.
$$ Then for every $g_i$, we have the complex
\begin{equation}\label{SpnsrCmplx}
  0\to g_i\xrightarrow{\partial_{i,0}}g_{i-1}\otimes V_p^*
  \xrightarrow{\partial_{i-1,1}}g_{i-2}\otimes\wedge^2 V_p^*
  \xrightarrow{\partial_{i-2,2}} 0\,,
\end{equation}
where the operators $\partial_{k,l} : g_k\otimes\wedge^l V_p^*\to
g_{k-1}\otimes\wedge^{l+1} V_p^*$ are defined in the following
way: any element $\xi\in g_k\otimes\wedge^l V_p^*$ can be
considered as an exterior form on $V_p$ with values in $g_k$, then
$$
  (\,\partial_{k,l}(\xi)\,)(v_1,\ldots,v_{l+1})
  =\sum_{i=1}^{l+1}(-1)^{i+1}[\,v_i\,,\;\xi(v_1,\ldots,\hat
  v_i,\ldots,v_{l+1})\,]\,.
$$ We denote by $H_p^{k,l}$ the cohomology group of this complex
in the term $g_k\otimes\wedge^l V_p^*$. It is called a {\it
Spenser cohomology group}.

%.................................................................
\subsubsection{The decomposition of tangent spaces}
Let $\theta_{k+1}\in J^{k+1}\pi$, let
$\theta_k=\pi_{k+1,k}(\theta_{k+1})$, and let
$[S]^{k+1}_p=\theta_{k+1}$. Then the tangent space to the image of
the section $j_kS$ at the point $\theta_k$ is defined by
$\theta_{k+1}$. We denote this tangent space by
$\Hr_{\theta_{k+1}}$. We have the following direct sum
decomposition of the tangent space to $J^k\pi$ at the point
$\theta_k$ $$
  T_{\theta_k}J^k\pi=\Hr_{\theta_{k+1}}\oplus
  T_{\theta_k}\bigl(\pi^{-1}(p)\bigr)\,.
$$

Let $X$ be a vector field in the base of $\pi$ defined in a
neighborhood of $p$. Then the value $X^{(k)}_{\theta_k}$ of
$X^{(k)}$ at the point $\theta_k$ has a unique decomposition
\begin{equation}\label{VctFldLftk}
  X^{(k)}_{\theta_k}=\Hr_{\theta_{k+1}}X^{(k)}+
  V_{\theta_{k+1}}X^{(k)}\,,
\end{equation}
where $\Hr_{\theta_{k+1}}X^{(k)}\in\Hr_{\theta_{k+1}}$ and
$V_{\theta_{k+1}}X^{(k)}\in T_{\theta_k}\bigl(\pi^{-1}(p)\bigr)$.
It follows from \eqref{InftPrlng} and \eqref{kPrlng} that if
$X=X^1\partial/\partial x^1+X^2\partial/\partial x^2$, then
\begin{equation}\label{Hrs-Evol}
  \Hr_{\theta_{k+1}}X^{(k)}=X^1D_1^{\theta_{k+1}}
  +X^2D_2^{\theta_{k+1}}\,,\quad
  V_{\theta_{k+1}}X^{(k)}=\re_{\psi(X)}^{\theta_{k+1}}\,,
\end{equation}
where
\begin{gather}
  D_j^{\theta_{k+1}}=\frac{\partial}{\partial x^j}
      +\sum_{0\leq|\sigma|\leq k}\sum_{i=0}^{3}u^i_{\sigma j}(\theta_{k+1})
      \frac{\partial}{\partial u^i_{\sigma}}\,,\notag\\
 \re^{\theta_{k+1}}_{\psi(X)}=\sum_{0\leq|\sigma|\leq k}\sum_{i=0}^{3}
 \Bigl(\,D_{\sigma}\bigl(\,\psi^i(X)\,\bigr)\,\Bigr)(\theta_{k+1})
 \frac{\partial}{\partial u^i_{\sigma}}\,.\label{Evol1}
\end{gather}

It follows from \eqref{DfrmtnVlst1} that the value
$X^{(k)}_{\theta_k}$ of the vector field $X^{(k)}$ at the point
$\theta_k$ is depended on the jet $[X]^{k+2}_p$.

%=================================================================
\subsection{Isotropy algebras}
Let $\theta_k\in J^k\pi$ and $p=\pi(\theta_k)$. By $G_{\theta_k}$
we denote the {\it isotropy group} of $\theta_k$, that is $$
  G_{\theta_k}=\bigl\{\;[f]^{2+k}_p\;\bigl|\;f\in\Gamma\,,\;
  f^{(k)}(\theta_k)=\theta_k\;\bigr\}
$$

By $\g_{\theta_k}$ we denote the Lie algebra of $G_{\theta_k}$. It
can be considered as a Lie subalgebra in $L^0_p/L^{2+k}_p$: $$
  \g_{\theta_k}=\bigl\{\;[X]^{2+k}_p\in L^0_p/L_p^{2+k}\;\bigl|\;
  X\in\Vect\R^2\,,\; X^{(k)}_{\theta_k}=0\;\bigr\}
$$ The subalgebra $\g_{\theta_k}\subset L^0_p/L_p^{2+k}$ is called
the {\it isotropy algebra} of $\theta_k$.

From this definition and \eqref{VctFldLftk}, \eqref{Hrs-Evol}, and
\eqref{Evol1}, we get
\begin{proposition}
  $[X]^{2+k}_p\in\g_{\theta_k}$ iff it is a solution of the system
  of linear algebraic equations
  \begin{equation}\label{IstrpAlgk}
    \bigl(\,D_{\sigma}(\,\psi^i_{X}\,)\,\bigr)(\theta_k)=0\,,\quad
    0\leq|\sigma|\leq k\,.
  \end{equation}
\end{proposition}
(We write $D_{\sigma}(\,\psi^i_{X}\,)\,)(\theta_k)$ in
\eqref{IstrpAlgk} instead
$D_{\sigma}(\,\psi^i_{X}\,)\,)(\theta_{k+1})$ because from $X_p=0$
we have that system \eqref{IstrpAlgk} depends on $\theta_k$ and it
is independent of $\theta_{k+1}$.)
\bigskip

Let $\theta_0\in J^0\pi$ and $p=\pi(\theta_0)$. From
\eqref{IstrpAlgk}, we get that the isotropy algebra
$\g_{\theta_0}$ of the point $\theta_0$ is defined by the
equations
\begin{equation}\label{IstrAlg}\left\{
  \begin{aligned}
    - 2u^0X^1_1 + u^0X^2_2 - u^1X^2_1
    &+ X^2_{11}=0\\
    - 3u^0X^1_2 - u^1X^1_1 - 2u^2X^2_1
    &- X^1_{11} + 2X^2_{12}=0\\
    - 2u^1X^1_2 - u^2X^2_2 - 3u^3X^2_1
    &- 2X^1_{12} + X^2_{22}=0\\
    -  u^2X^1_2 + u^3X^1_1 - 2u^3X^2_2
    &- X^1_{22}=0
  \end{aligned}\right.
\end{equation}

It follows from \eqref{IstrAlg} that
$$
  \rho_{2,1}(\,\g_{\theta_0}\,)=L^0_p/L^1_p\,.
$$

Let
$$
  g_{\theta_0}=\g_{\theta_0}\cap\;(L^1_p/L^2_p)\,.
$$
Obviously, it is a commutative subalgebra in $\g_{\theta_0}$.
From \eqref{IstrAlg}, we get that $g_{\theta_0}$ is defined by the
equations
\begin{equation}\label{Smbl}\left\{
  \begin{aligned}
     X^2_{11}=0\\
     X^1_{11} - 2X^2_{12}=0\\
     2X^1_{12} - X^2_{22}=0\\
     X^1_{22}=0
  \end{aligned}\right.
\end{equation}

It is clear that $g_{\theta_0}$ and $g_{\tilde\theta_0}$ are
canonically isomorphic for any $\theta_0,\tilde\theta_0\in
J^0\pi$. Therefore we shall write $g$ instead $g_{\theta_0}$.

It follows from \eqref{Smbl} that
\begin{equation}\label{Dim_g}
  \dim g=2
\end{equation}
and we can choose
\begin{equation}\label{Gnrtrs_g2}
  \begin{aligned}
    e_1&=2\frac{\partial}{\partial x^1}\otimes (dx^1\odot dx^1)
      +\frac{\partial}{\partial x^2}\otimes (dx^1\odot dx^2)\,,\\
    e_2&=2\frac{\partial}{\partial x^2}\otimes (dx^2\odot dx^2)
     +\frac{\partial}{\partial x^1}\otimes (dx^1\odot dx^2)
  \end{aligned}
\end{equation}
as independent generators of $g$.

It is easy to check that the 1--prolongation $g^{(1)}$ of $g$ is
trivial, that is
\begin{equation}\label{FstPrlng}
  g^{(1)}=\{0\}\,.
\end{equation}

%=================================================================
\subsection{Isotropy spaces}
By definition, put
\begin{gather}\label{IstrpSpc}
  \A_{\theta_{k+1}}=\bigl\{\;[X]^{2+k}_p\in W_p/L^{2+k}_p\;\bigr|
  \;X^{(k)}_{\theta_k}\in\Hr_{\theta_{k+1}}\;\bigr\}\,,\\
  k=0,1,\ldots,\infty\,.\notag
\end{gather}
From \eqref{VctFldLftk}, \eqref{Hrs-Evol}, and \eqref{Evol1}, we
get
\begin{proposition}\label{IstrpSpcEqv}
  $[X]^{2+k}_p\in\A_{\theta_{k+1}}$ iff $[X]^{2+k}_p$ is a
  solution of the system of linear equations
  \begin{equation}\label{IstrpSpck+1}
    \bigl(\,D_{\sigma}(\,\psi^i_{X}\,)\,\bigr)(\theta_{k+1})=0\,,\quad
    0\leq|\sigma|\leq k\,.
  \end{equation}
\end{proposition}

We say that $\A_{\theta_{k+1}}$ is the {\it isotropy spase} of
$\theta_{k+1}$.
\begin{theorem}\label{PrpIstrpSpc}
  \begin{enumerate}
    \item $\rho_{k+2,k+1}(\A_{\theta_{k+1}})\subset \A_{\theta_k}$.
    \item $[\,\cdot\,,\,\cdot\,]: \A_{\theta_{k+1}}\times
      \A_{\theta_{k+1}}\to \A_{\theta_k}$.
  \end{enumerate}
\end{theorem}
\begin{proof} The first statement is obvious.

  Prove the second one. Let
  $[X]^{2+k}_p,[Y]^{2+k}_p\in\A_{\theta_{k+1}}$, let
  $\theta_{\infty}\in\pi_{\infty}^{-1}(p)\,,\linebreak
  \theta_k=\pi_{\infty,k}(\theta_{\infty})\,,\;
  \theta_{k-1}=\pi_{k,k-1}(\theta_k)$,
  and let
  $$
    X = X^1\frac{\partial}{\partial x^1}
      + X^2\frac{\partial}{\partial x^2}\,,\;
    Y = Y^1\frac{\partial}{\partial x^1}
      + Y^2\frac{\partial}{\partial x^2}\,.
  $$
  Then
  $$
    \bigl[\,[X]^{2+k}_p,[Y]^{2+k}_p\,\bigr]
    =\bigl[\,[\,X\,,Y\,]\,\bigr]^{2+k-1}_p
  $$
  and
  $$
    \bigl[\,[\,X\,,Y\,]\,\bigr]^{2+k-1}_p\in\A_{\theta_k}\quad
    \text{iff}\quad
    [\,X,\,Y\,]^{(k-1)}_{\theta_{k-1}}\in\Hr_{\theta_k}\,.
  $$
  We have
  \begin{multline*}
    [\,X,\,Y\,]^{(k-1)}_{\theta_{k-1}}
    =(\pi_{\infty,k-1})_*\bigl[\,X,\,
    Y\,\bigr]^{(\infty)}_{\theta_{\infty}}
    =(\pi_{\infty,k-1})_*\bigl[\,X^{(\infty)},\,
    Y^{(\infty)}\,\bigr]_{\theta_{\infty}}\\
    =(\pi_{\infty,k-1})_*\bigl[\,X^1D_1+X^2D_2+\re_{\psi(X)},\;
    Y^1D_1+ Y^2D_2+\re_{\psi(Y)}\,\bigr]_{\theta_{\infty}}\,.
  \end{multline*}
  Taking into account the well known relations (see \cite{KV})
  $$
    [\,D_1,\,D_2\,]=[\,D_1,\,\re_{\psi}\,]=[\,D_2,\,\re_{\psi}\,]
    =0\quad\text{and}\quad[\,\re_{\phi},\,\re_{\psi}\,]=
    \re_{\{\phi,\psi\}}\,,
  $$
  where $\{\phi,\psi\}=\re_{\phi}(\psi)-\re_{\psi}(\phi)$, we get
   \begin{multline*}
    [\,X,\,Y\,]^{(k-1)}_{\theta_{k-1}}
    =(\pi_{\infty,k-1})_*
    \Bigl(\,(X^1Y^1_1+X^2Y^1_2-Y^1X^1_1-Y^2X^1_2)D_1+\\
    +(X^1Y^2_1+X^2Y^2_2-Y^1X^2_1-Y^2X^2_2)D_2
    +[\,\re_{\psi(X)}\,,\,\re_{\psi(Y)}\,]
    \,\Bigr)_{\theta_{\infty}}=\\
    =\Hr_{\theta_k}[\,X,\,Y\,]^{(k-1)}
    +\re^{\theta_k}_{\{\psi(X),\psi(Y)\}}\,.
  \end{multline*}
From \eqref{DfrmtnVlst1}, we obtain
\begin{multline*}
  \{\psi(X),\psi(Y)\}^i
  =\psi^{i'}(X)\frac{\partial\psi^i(Y)}{\partial u^{i'}}
   +D_j(\psi^{i'}(X))\frac{\partial\psi^i(Y)}{\partial u^{i'}_j}\\
   -\psi^{i'}(Y)\frac{\partial\psi^i(X)}{\partial u^{i'}}
   -D_j(\psi^{i'}(Y))\frac{\partial\psi^i(X)}{\partial
   u^{i'}_j}\,.
\end{multline*}
  From \eqref{Evol1}, we get now that
  $\re^{\theta_k}_{\{\psi(X),\psi(Y)\}}=0$.
\end{proof}

%%%%%%%%%%%%%%%%%%%%%%%%%%%%%%%%%%%%%%%%%%%%%%%%%%%%%%%%%%%%%%%%%%
\section{Differential invariants}
%=================================================================
\subsection{Horizontal subspaces}
We shall say that a 2-dimensional subspace $H\subset W_p/L_p^k$ is
{\it horisontal} if $$
  \rho_{k}(H)=V_p\,.
$$

Let $\theta_k\in J^k\pi$ and
$\theta_{k+1}\in\pi^{-1}_{k+1,k}(\theta_k)$; then it is clear that
\begin{equation}\label{Invlv}
  \g_{\theta_k}\subset\A_{\theta_{k+1}}\;\;
  \forall\;\theta_{k+1}\in\pi_{k+1,k}^{-1}(\theta_k)\,.
\end{equation}
It is obvious that a 2-dimensional subspace
$H\subset\A_{\theta_{k+1}}$ is horizontal iff $$
  \A_{\theta_{k+1}}=H\oplus\g_{\theta_k}\,.
$$ Any two horizontal subspaces $H,\tilde
H\subset\A_{\theta_{k+1}}$ define the linear function $$
  f_{H,\tilde H}:V_p\to\g_{\theta_k}\,,\quad
  f_{H,\tilde H}: X\mapsto (\rho_{k+2}|_H)^{-1}(X)
  -(\rho_{k+2}|_{\tilde H})^{-1}(X)\,.
$$ It is clear that for any horizontal subspace
$H\subset\A_{\theta_{k+1}}$ and for any linear function
$f:V\to\g_{\theta_k}$, there exist a unique horizontal subspace
$\tilde H\subset\A_{\theta_{k+1}}$ with $f=f_{H,\tilde H}$.
\bigskip

Further in this subsection, we shall investigate horizontal
subspaces of $\A_{\theta_1}$.

By $H_p$ we denote the horizontal subspace in $W_p/L^1_p$
generated by constant vector fields.

By $H_{\theta_1}$ we denote a horizontal subspace in
$\A_{\theta_1}$ with
\begin{equation}\label{HrzntSpc0}
  \rho_{2,1}(H_{\theta_1})=H_p\,.
\end{equation}
From $$
  \rho_{2,1}(\A_{\theta_1})=W_p/L^1_p\,,
$$ we have that horizontal subspaces $H_{\theta_1}$ exist.
Obviously, $H_{\theta_1}$ is defined by
\begin{equation}\label{HrzntSpc1}
  H_{\theta_1}=\bigl\{\;[X]^2_p=(\,X^i,0,X^i_{\sigma}\,)
  \,,\; i=1,2\,,\;|\sigma|=2\;\bigr\}
\end{equation}
in the standard coordinates.

It is clear now that for any two horizontal subspaces
$H_{\theta_1},\tilde H_{\theta_1}$ satisfying to
\eqref{HrzntSpc0}, we get $$
  f_{H_{\theta_1},\tilde H_{\theta_1}}:V_p\to g\,.
$$

Taking into account that $g\neq \{0\}$, we obtain that there exist
a lot of horizontal subspaces satisfying to \eqref{HrzntSpc0}. We
choose one of them in the following way.

A horizontal subspace $H_{\theta_1}$ defines the form
$\omega_{H_{\theta_1}}\in L_p^0/L_p^1\otimes\wedge^2V_p^*$ by the
formula $$
  \omega_{H_{\theta_1}}(X,Y)=
  \bigl[\,(\rho|_{H_{\theta_1}})^{-1}(X),
  (\rho|_{H_{\theta_1}})^{-1}(Y)\,\bigr]\quad\forall\,X,Y\in V_p\,.
$$ From the Spenser complex
\begin{equation}\label{SpnsrCmplx1}
 0\to g^{(1)}\xrightarrow{\partial_{3,0}} g\otimes
 V_p^*\xrightarrow{\partial_{2,1}}
 L^0_p/L^1_p\otimes
 \wedge^2V_p^*\xrightarrow{\partial_{1,2}} 0\,,
\end{equation}
we get that $\omega_{H_{\theta_1}}$ defines the Spenser cohomology
class $\{\omega_{H_{\theta_1}}\}\in H_p^{1,2}$.
\begin{proposition} The cohomology class $\{\omega_{H_{\theta_1}}\}$ is
  trivial.
\end{proposition}
\begin{proof} From \eqref{FstPrlng} we get that $\partial_{2,1}$
  is an injection in \eqref{SpnsrCmplx1}. From \eqref{Dim_g}, we
  obtain $\dim g\otimes V_p^*=4$. Obviously,
  $\dim L^0_p/L^1_p\otimes\wedge^2V_p^*=4$. As a result, we
  obtain $\im\partial_{2,1}=\ker\partial_{1,2}$ in
  \eqref{SpnsrCmplx1}.
\end{proof}
\begin{corollary} There exists a unique horizontal subspace
$H_{\theta_1}\subset\A_{\theta_1}$ with $\omega_{H_{\theta_1}}=0$.
\end{corollary}
\begin{proof} Prove the uniqueness. Suppose
$H_{\theta_1},\,\tilde H_{\theta_1}$ are horizontal sub\-spa\-ces
of $\A_{\theta_1}$ with $\omega_{H_{\theta_1}}=\omega_{\tilde
H_{\theta_1}}=0$. We have $\omega_{H_{\theta_1}}=\omega_{\tilde
H_{\theta_1}} +\partial_{2,1}(f_{H_{\theta_1},\tilde
H_{\theta_1}})$. Therefore, $\partial_{2,1}(f_{H_{\theta_1},\tilde
H_{\theta_1}})=0$. Taking into account that $\partial_{2,1}$ is an
injection, we get that $f_{H_{\theta_1},\tilde H_{\theta_1}}=0$.
This means that $H_{\theta_1}=\tilde H_{\theta_1}$

Prove the existence. We have $\{\omega_{H_{\theta_1}}\}=\{0\}$.
Therefore there exist $h\in g\otimes V_p^*$ with
$\omega_{H_{\theta_1}}=\partial_{2,1}(h)$. It follows that the
horizontal subspace $$
  \tilde H_{\theta_1}=\bigl\{\,(\rho_2|_{H_{\theta_1}})^{-1}(X)-h(X)\,,
  \;\;X\in V_p\,\bigr\}
$$ satisfies to the property $\omega_{\tilde H_{\theta_1}}=0$.
\end{proof}

Now, we express the horizontal space $H_{\theta_1}$ with
$\omega_{H_{\theta_1}}=0$ in terms of standard coordinate $x^1,
x^2, u^i(\theta_1), u^i_j(\theta_1)$. Let $$
  \bigl(\rho_2|_{H_{\theta_1}}\bigr)^{-1}(X)=(\,X^i,\;0\,,\;f^i_{jk,r}X^r\,)
  \,,\;\;\forall\,X\in V_p\,.
$$ Then the property $\omega_{H_{\theta_1}}=0$ means that
\begin{equation}\label{Smmtr}
  f^i_{jk,r}=f^i_{jr,k}\,.
\end{equation}
From proposition \ref{IstrpSpcEqv} we obtain that elements
$(\,X^i,\;0\,,\;f^i_{jk,r}X^r\,)\in H_{\theta_1}$ is a solutions
of the system $$
  \left\{\begin{aligned}
    -u^0_1X^1-u^0_2X^2 + f^2_{11,r}X^r=0\\
    -u^1_1X^1-u^1_2X^2 - f^1_{11,r}X^r + 2f^2_{12,r}X^r=0\\
    -u^2_1X^1-u^2_2X^2 - 2f^1_{12,r}X^r + f^2_{22,r}X^r=0\\
    -u^3_1X^1-u^3_2X^2 - f^1_{22,r}X^r=0
  \end{aligned}\right.
$$ From this system and \eqref{Smmtr}, we obtain
\begin{equation}\label{H_theta_1}
  \left\{\begin{aligned}
    f^2_{11,1}=u^0_1\,,\;\;f^2_{11,2}=f^2_{12,1}=u^0_2\,,\\
    f^2_{12,2}=f^2_{22,1}=\frac{1}{3}(\,2u^1_2-u^0_1\,)\,,\\
    f^2_{22,2}=-2u^3_1+u^2_2\,,\\
    f^1_{22,2}=-u^3_2\,,\;\;f^1_{22,1}=f^1_{12,2}=-u^3_1\,,\\
    f^1_{12,1}=f^1_{11,2}=\frac{1}{3}(\,u^1_2-2u^2_1\,)\,,\\
    f^1_{11,1}=2u^0_2-u^1_1\,.
  \end{aligned}\right.
\end{equation}

%=================================================================
\subsection{The obstruction form}\label{SubSctObstrFrm}
Let $\theta_2\in J^2\pi$ and $\theta_1=\pi_{2,1}(\theta_2)$. It is
not difficult to prove that
\begin{equation}\label{Srjctn}
  \rho_{3,2}(\A_{\theta_2})=\A_{\theta_1}.
\end{equation}
Let $H_{\theta_1}$ be the horizontal subspace of $\A_{\theta_1}$
with $\omega_{H_{\theta_1}}=0$. From \eqref{Srjctn} and
\eqref{FstPrlng}, we get that there exist a unique horizontal
subspace $H_{\theta_2}\subset\A_{\theta_2}$ with
\begin{equation}\label{HrzntSpc2}
  \rho_{3,2}(H_{\theta_2})=H_{\theta_1}\,.
\end{equation}

It follows from item (2) of theorem \ref{PrpIstrpSpc} that
$H_{\theta_2}$ defines the 2--form
$\omega_{\theta_2}\in\A_{\theta_1}\otimes\wedge^2V_p^*$ by the
formula $$
  \omega_{\theta_2}(X,Y)=
  \bigl[\,(\rho_3|_{H_{\theta_2}})^{-1}(X),\,
  (\rho_3|_{H_{\theta_2}})^{-1}(Y)\,\bigr]\quad\forall\,X,Y\in V_p\,.
$$ From $\omega_{H_{\theta_1}}=0$ we obtain $$
  \omega_{\theta_2}\in g\otimes(V_p^*\wedge V_p^*)
$$

Now we can define the horizontal differential 2-form
$\omega^{(2)}$ on $J^2\pi$ with values in $g$ by the following
formula
\begin{equation}\label{ObstrtnFrm}
  \omega^{(2)} : \theta_2\longmapsto\pi_2^*(\omega_{\theta_2})\,.
\end{equation}

Obviously, $H_{\theta_2}$ is defined by $$
  H_{\theta_2}=\bigl\{\;[X]^2_p=(\,X^i,\;0,\;f^i_{j_1j_2,r}X^r,\;
  f^i_{j_1j_2j_3,r}X^r\,)
  \;\bigr\}
$$ in the standard coordinates. Hence, $$
  \omega_{\theta_2}=2f^i_{j_1j_2[k,r]}\bigl(\frac{\partial}{\partial x^i}
  \otimes (dx^{j_1}\odot dx^{j_2})\bigr)\otimes\,(dx^k\wedge\,dx^r)\,.
$$ Taking into account \eqref{Dim_g} and \eqref{Gnrtrs_g2}, we get
\begin{equation}\label{Obstrtn}
  \omega^{(2)}=(\;F^1\cdot e_1+F^2\cdot e_2\;)
  \otimes\,(dx^1\wedge\,dx^2)\,,
\end{equation}
where $F^1=f^1_{11[1,2]}$ and $F^2=f^2_{22[1,2]}$.

Calculate the functions $F^1\,,\;F^2$. From proposition
\ref{IstrpSpcEqv} we obtain that elements $$
  (\,X^i,\;0,\;f^i_{j_1j_2,r}X^r,\;
  f^i_{j_1j_2j_3,r}X^r\,)\in H_{\theta_2}
$$ is a solutions of system \eqref{IstrpSpck+1} for $k=1$. From
this system and \eqref{H_theta_1}, we get
\begin{multline}\label{F1}
  F^1= 3u^0_{22}-2u^1_{12}+u^2_{11}\\
      +3u^3u^0_1-3u^2u^0_2+
      2u^1u^1_2-u^1u^2_1-3u^0u^2_2+6u^0u^3_1\,,
\end{multline}
\begin{multline}\label{F2}
    F^2= u^1_{22}-2u^2_{12}+3u^3_{11}\\
    -3u^0u^3_2+3u^1u^3_1
       - 2u^2u^2_1+u^2u^1_2+3u^3u^1_1-6u^3u^0_2\,.
\end{multline}
Note that first the coefficients $F^1$ and $F^1$ were obtained by
Cartan in \cite{Crtn} as unique nonzero coefficients of the
curvature form of the projective connection corresponding to
equation \eqref{eq}.

Thus we obtain the following expession of $\omega^{(2)}$ in the
standard coordinates
\begin{multline}\label{Obstrtn1}
  \omega^{(2)}=\Bigl(F^1\bigl(2\frac{\partial}{\partial x^1}
  \otimes (dx^1\odot dx^1)
  +\frac{\partial}{\partial x^2}\otimes (dx^1\odot dx^2)\bigr)\\
  +F^2\bigl(2\frac{\partial}{\partial x^2}\otimes (dx^2\odot dx^2)
  +\frac{\partial}{\partial x^1}\otimes
  (dx^1\odot dx^2)\bigr)\Bigr)\\
  \otimes\,(dx^1\wedge\,dx^2)\,,
\end{multline}
where $F^1$ and $F^2$ are defined by \eqref{F1} and \eqref{F2}
respectively.
\bigskip

We recall, that a differential form defined on $J^k\pi$ is {\it a
differential invariant of the action of $\Gamma$ on $\pi$} if it
is invariant w.r.t. the pseudogroup $\Gamma^{(k)}$.
\begin{theorem}\label{DffInv} The form $\omega^{(2)}$ is a
  differential invariant of the action of $\Gamma$ on $\pi$.
\end{theorem}
\begin{proof} Let $f\in\Gamma$, let $p$ be a point from the domain
  of $f$, and let $\theta_2\in J_p^2\pi$. We should check that
  \begin{equation}\label{InvAct}
    (f^{(2)})^*\Bigl(\omega^{(2)}\bigr|_{f^{(2)}(\theta_2)}\Bigr)
    =\omega^{(2)}\bigr|_{\theta_2}\,.
  \end{equation}
  We shall check it in the standard coordinates. It is clear that the
  left side of \eqref{InvAct} is depend of $[f]^4_p$. This jet can
  be represented in the following way
  $$
    [f]^4_p=[f_1]^4_p\cdot[f_2]^4_p\,,
  $$
  where $[f_1]^4_p$ is jet of the affine transformation and
  $[f_2]^1_p=[\id]^1_p$\,. It can easily be checked that
  $\omega^{(2)}$ is invariant w.r.t. affine transformations.
  Therefore it remains to check that equation \eqref{InvAct} holds
  for an arbitrary point transformation $f$ with
  $[f]^1_p=[\id]^1_p$. Taking into account that
  $\omega^{(2)}$ is horizontal, we get that equation
  \eqref{InvAct} holds for a point transformation $f$ with
  $[f]^1_p=[\id]^1_p$ iff
  $$
    F^1\bigl(f^{(2)}(\theta_2)\bigr)=F^1(\theta_2)\,,\quad
    F^2\bigl(f^{(2)}(\theta_2)\bigr)=F^2(\theta_2)
  $$
  It is clear that the last equations hold iff the restrictions of
  $F^1\,,\;F^2$ to $J_p^2\pi$ are 1-st integrals for any vector
  field $\xi^{(2)}\bigr|_{J_p^2\pi}$ with
  $[\xi]^{\infty}_p\in L^1_p$. The last statement about $F^1$ and
  $F^2$ can be easy checked by direct calculations in standard
  coordinates.
\end{proof}
\bigskip

In his paper \cite{Crtn}, Cartan proved that equation \eqref{eq}
can be reduced to the linear form by a point transformation iff
the collection of its coefficients is a solution of the system of
PDEs
\begin{equation}\label{DfM}
  F^1=0\,,\quad F^2=0\,.
\end{equation}
This means that $\omega^{(2)}$ is a unique obstruction to the
linearizability of equations \eqref{eq} by point transformations.

Below, we give the independent proof of this fact.

Let $$
  M=\Bigl\{\;\theta_2\in J^2\pi\;\Bigl|\;
  \omega^{(2)}\bigr|_{\theta_2}=0\;\Bigr\}\,.
$$ From \eqref{Obstrtn}, it follows that $M$ is defined by system
of algebraic equations \eqref{DfM}.

By $\0$ we denote the zero section of $\pi$, by $0_k$ we denote
$[\0]^k_0\,,\, k=0,1,2,\ldots$
\begin{lemma}\label{ZeroStrFrm} $M=\Orb_{0_2}$.
\end{lemma}
\begin{proof} It is clear that $\dim\Orb_{0_2}
  =\dim W_0/L_0^4 -\dim\g_{0_2}$. We have \linebreak
  $\dim W_0/L_0^4=30$. It
  is easy to calculate that $\dim\g_{0_2}=6$. Therefore
  $\dim\Orb_{0_2}=24$. From \eqref{DfM}, we have that $\dim M
  =\dim J^2\pi-2=24$ too.

  Obviously, $\omega^{(2)}\bigl|_{0_2}=0$. Now from
  theorem \ref{DffInv}, we get that $\Orb_{0_2}\subset M$. At
  last, the sets $M$ and $\Orb_{0_2}$ are connected subsets in
  $J^2\pi$. This concludes the proof.
\end{proof}
\begin{lemma}\label{ZeroStrFrm1} Let $\theta_2\in\Orb_{0_2}$
  and let $\theta_1=\pi_{2,1}(\theta_2)$; then the natural
  projection of the isotropy groups of these points
  $$
    G_{\theta_2}\to G_{\theta_1}\,,\quad
    [f]^4_p\mapsto [f]^3_p\,,
  $$
  is a bijection.
\end{lemma}
\begin{proof}
  It is easy to prove that the natural projection
  $G_{0_2}\to G_{0_1}$ is an injection and that $\dim G_{0_1}
  =\dim G_{0_2}=6$. Therefore the natural projection
  $G_{0_2}\to G_{0_1}$ is a bijection. The projection
  $\pi_{2,1}:\Orb_{0_2}\to J^1\pi$ is a surjective. This implies
  the proof.
\end{proof}
For any section $S$ of $\pi$, by $\omega^{(2)}_S$ we denote the
form $\bigl(j_2S\bigr)^*(\omega^{(2)})$.
\begin{theorem}\label{Obstrctn} The section $S$ can be transformed (locally) to
  $\0$ by a point transformation iff $\omega^{(2)}_S\equiv 0$.
\end{theorem}
\begin{proof} The necessity is obvious.

  Prove the sufficiency. To this end, we should prove that the
  system of PDEs w.r.t. an unknown point transformation $f$
  $$
    \0=f^{(0)}\circ S\circ f^{-1}
  $$
  has a solution. By $\E(\0,S)$ we denote this system. It easy to
  prove that the symbol of this PDE system at any point is the
  same as the subalgebra $g$ defined above by \eqref{Smbl}. From
  \eqref{FstPrlng}, we obtain that the first prolongation
  $\E^{(1)}(\0,S)$ of $\E(\0,S)$ has the zero
  symbol at every point. Therefore $\E^{(1)}(\0,S)$ has
  a solution if the natural projection $\E^{(2)}(\0,S)\to
  \E^{(1)}(\0,S)\,,\;[f]_p^4\mapsto [f]_p^3$, is a surjection
  (see \cite{Krnsh}).

  Let us check that this projection is a surjection.
  Let $[f]^3_p\in\E^{(1)}(\0,S)$. It takes $[S]^1_p$ to
  $[\0]^1_{f(p)}$. By assumption, $\omega^{(2)}([S]^2_p)=0$. It
  follows from lemma
  \ref{ZeroStrFrm} that $[S]^2_p\in\Orb_{0_2}$. Obviously,
  $[\0]^2_{f(p)}\in\Orb_{0_2}$ too. Hence there exist a
  point transformation $f'$ such that its jet $[f']^4_p$ takes
  $[S]^2_p$ to $[\0]^2_{f(p)}$. This means that
  $[f']^4_p\in\E^{(2)}(\0,S)\,,\;\;[f']^3_p\in\E^{(1)}(\0,S)$, and
  $[f']^3_p$ takes $[S]^1_p$ to $[\0]^1_{f(p)}$. From
  the last, we obtain that there exist
  $g\in G_{[S]^1_p}$ with $[f']^3_p\cdot g=[f]^3_p$. From lemma
  \ref{ZeroStrFrm1}, we get that there
  exist $g'\in G_{[S]^2_p}$ with $\rho_{4,3}(g')=g$. Obviously,
  $[f']^4_p\cdot g'\in\E^{(2)}(\0,S)$ and it is clear that
  $[f]^3_p$ is the image of $[f']^4_p\cdot g'$ under the natural
  projection $\E^{(2)}(\0,S)\to\E^{(1)}(\0,S)$. Thus, this natural
  projection is a surjection.
\end{proof}
\begin{corollary} The form $\omega^{(2)}$ is a unique obstruction
  to the linearizability of ODEs \eqref{eq} by point
  transformations.
\end{corollary}
\begin{proof} It is well known that any two 2--order linear ODEs
  are (locally) equivalent w.r.t. point transformations. This
  implies the proof.
\end{proof}

%%%%%%%%%%%%%%%%%%%%%%%%%%%%%%%%%%%%%%%%%%%%%%%%%%%%%%%%%%%%%%%%%%
%

%

\begin{thebibliography}{99}
\bibitem{Crtn} E. Cartan, {\it Sur les varietes a connexion
  projective}, Bull. Soc. Math. France 52 (1924), 205 -- 241.
\bibitem{GYum} V.N.Gusyatnikova, V.A.Yumaguzhin, {\it Point
  transformations and linearisability of 2-order ordinary
  differential equations}, Matemeticheskie Zametki Vol. 49, No. 1,
  pp. 146 - 148, 1991 (in Russian).
\bibitem{Strnbrg} S. Sternberg, {\it Lectures on Differential
  Geometry}, New Jersy, Prentice-Hall, Inc., 1964.
\bibitem{Arnd} V.I.Arnold, {\it Advanced chapters of the theory of
  ordinary differential equations}, Nauka, Moskow, 1978 (in Russian).
\bibitem{KV} I.S.Krasil'shchik, A.M.Vinogradov, Editors, {\it
  Symmetries and conservation laws for differential equations of
  mathematical Physics}, Translations of Mathematical Monographs.
  Vol.182, Providence RI: American Mathematical Society, 1999.
\bibitem{BrnshtnRznfld} I.N.Bernshtein, B.I.Rozenfel'd, {\it
  Homogeneous spaces of infinitely demension Lie algebras and
  characteristic classes of foliations}, Uspekhi Matematicheskikh
  Nauk, Vol. 28, No. 4, pp. 103-138, 1973.(in Russian).
\bibitem{GllmnStrnbrg} V.Guillemin, S.Sternberg, {\it An algebraic
  model of transitive differential geometry}, Bull. Amer. Math. Soc.,
  vol. 70, (1964), pp. 16-47.
\bibitem{Krnsh} M.Kuranishi, {\it Lectures on involutive systems of
  partial differential equations}, S\~ao Paulo, 1967.
\end{thebibliography}
\end{document}